\documentclass[11pt]{amsart}
\usepackage{amscd,amsmath,latexsym,amsthm,amsfonts,amssymb,graphicx,color,geometry,hyperref,tikz}
\usepackage[utf8]{inputenc}

\providecommand{\U}[1]{\protect\rule{.1in}{.1in}}
\providecommand{\U}[1]{\protect\rule{.1in}{.1in}}

\newtheorem{theorem}{Theorem}[section]
\newtheorem{proposition}[theorem]{Proposition}

\newtheorem{lemma}[theorem]{Lemma}
\newtheorem{problem}[theorem]{Problem}

\newtheorem{conjecture}[theorem]{Conjecture}
\numberwithin{equation}{section}

\begin{document}
\title[A Gale-Berlekamp permutation-switching problem]{A Gale-Berlekamp permutation-switching problem in higher dimensions}
\author{G. Ara\'ujo}
\address{Departamento de Matem\'{a}tica \\
\indent Universidade Estadual da Para\'{\i}ba \\
\indent 58.429-600 - Campina Grande, Brazil}
\email{gustavoaraujo@cct.uepb.edu.br; gdasaraujo@gmail.com}
\author{D. Pellegrino}
\address{Departamento de Matem\'{a}tica \\
\indent Universidade Federal da Para\'{\i}ba \\
\indent 58.051-900 - Jo\~{a}o Pessoa, Brazil}
\email{pellegrino0q.cnpq.br; dmpellegrino@gmail.com}
\thanks{D. Pellegrino is supported by R\'{e}sau Franco-Br\'{e}silian en Math\'{e}matiques}

\begin{abstract}
Let an $n\times n$ array $\left( a_{ij}\right) $ of lights be given, each either on (when $a_{ij}=1$) or off (when $a_{ij}=-1$). For each row and each column there is a switch so that if the switch is pulled ($x_{i}=-1$ for row $i$ and $y_{j}=-1$ for column $j$) all of the 	lights in that line are switched: on to off or off to on. The unbalancing lights problem (Gale-Berlekamp switching game)	consists in maximizing the difference between the lights on and off. We obtain the exact parameters for a generalization of the unbalancing lights problem in higher dimensions. 
\end{abstract}

\maketitle
\section{Introduction}

We begin by presenting a combinatorial game, sometimes called Gale-Berlekamp switching game or unbalancing lights problem (for a presentation we refer, for instance to the classical book of Alon and Spencer \cite{alon}). Let an $n\times n$ array $\left( a_{ij}\right) $ of lights be given, each
either on (when $a_{ij}=1$) or off (when $a_{ij}=-1$). Let us also suppose
that for each row and each column there is a switch so that if the switch is
pulled ($x_{i}=-1$ for row $i$ and $y_{j}=-1$ for column $j$) all of the
lights in that line are switched: on to off or off to on. The problem
consists in maximizing the difference between the lights on and off. 

A probabilistic approach (using the Central Limit Theorem) to this problem
(see \cite{alon}) provides the following asymptotic estimate:

\begin{theorem}[{\protect\cite[Theorem 2.5.1]{alon}}]
Let $a_{ij}=\pm1$ for $1\leq i,j\leq n$. Then there exist $x_{i},y_{j}=\pm1$%
, $1\leq i,j\leq n$, such that 
\begin{equation}
\sum\limits_{i,j=1}^{n}a_{ij}x_{i}y_{j}\geq\left( \sqrt{2/\pi}%
+o(1)\right)n^{3/2},  \label{6hh3}
\end{equation}
and the exponent $3/2$ is optimal. In other words, for any initial
configuration $\left(a_{ij}\right)$ it is possible to perform switches so
that the number of lights on minus the number of lights off is at least $%
\left( \sqrt{2/\pi}+o(1)\right) n^{3/2}$.
\end{theorem}

In higher dimensions (cf. mathoverflow.net/questions/59463/unbalancing-lights-in-higher-dimensions, by A. Montanaro) the unbalancing lights problem is stated as follows:

Let an $n\times\cdots\times n$ array $\left(a_{i_{1}\cdots i_{m}}\right)$ of
lights be given each either on (when $a_{i_{1}\cdots i_{m}}=1$) or off (when 
$a_{i_{1}\cdots i_{m}}=-1$). Let us also suppose that for each $i_{j}$ there
is a switch so that if the switch is pulled ($x_{i_{j}}=-1$) all of the
lights in that line are ``switched": on to off or off to on. The goal is to
maximize the difference between the lights on and off.

It is a well known consequence of the Bohnenblust--Hille inequality \cite{bhv} that
there exist $x_{i_{j}}^{(k)}=\pm1$, $1\leq j\leq n$ and $k=1,\ldots,m,$ and
a constant $C\geq1$, such that 
\begin{equation*}
\sum\limits_{i_{1},\ldots,i_{m}=1}^{n}a_{i_{1}\cdots
i_{m}}x_{i_{1}}^{(1)}\cdots x_{i_{m}}^{(m)}\geq\frac{1}{C^{m}}n^{\frac{m+1}{2%
}}
\end{equation*}
and that the exponent $\frac{m+1}{2}$ is sharp. A step further suggested by A. Montanaro is to
investigate if the term $C^{m}$ can be improved. Using recent estimates of
the Bohnenblust--Hille inequality (see \cite{adv}) it is plain that there
exist $x_{i_{j}}=\pm1$, $1\leq j\leq n$ and a constant $C>0$ such that 
\begin{equation}
\sum\limits_{i_{1},\ldots,i_{m}=1}^{n}a_{i_{1}\cdots
i_{m}}x_{i_{1}}^{(1)}\cdots x_{i_{m}}^{(m)}\geq\frac{1}{1.3m^{0.365}}n^{%
\frac{m+1}{2}},  \label{6hh2}
\end{equation}
and the exponent $\frac{m+1}{2}$ is sharp. It is still an open problem if
the term $1.3m^{0.365}$ (here and henceforth $1.3m^{0.365}$ is just a
simplification of $\kappa m^{\frac{2-\log2-\gamma}{2}},$ where $\gamma$ is
the Euler--Mascheroni constant) can be improved to a universal constant.

Some variants of the unbalancing lights problem have been already
investigated (see \cite{br}). In this paper we consider a more general
problem:

\begin{problem}
\label{ttt} Let $\left(a_{i_{1}\cdots i_{m}}\right) $ be an $%
n\times\cdots\times n$ array of (real or complex) scalars such that $%
\left\vert a_{i_{1}\cdots i_{m}}\right\vert =1.$ For $p\in\lbrack1,\infty]$,
maximize 
\begin{equation*}
g(p)=\left\{\sum\limits_{i_{1},\ldots,i_{m}=1}^{n}a_{i_{1}\cdots
i_{m}}x_{i_{1}}^{(1)}\cdots x_{i_{m}}^{(m)} \ : \ \left\Vert ( x_{i}^{(j)})
_{i=1}^{n}\right\Vert _{p}=1\text{ for all }j=1,\ldots,m\right\}.
\end{equation*}
\end{problem}

When $p=\infty$ with real norm-one scalars is precisely the classical unbalancing lights problem
in higher dimensions (\cite{montanaro}).

The main result of this paper, in particular, gives sharp exponents for the unbalancing lights problem for $p\geq 2$:

\begin{itemize}
\item  If $p\in\lbrack2,\infty]$, then 
\begin{equation}
g(p)\geq\frac{1}{1.3m^{0.365}}n^{\frac{mp+p-2m}{2p}}
\label{0000}
\end{equation}
and the exponents $\frac{mp+p-2m}{2p}$ are sharp. 

\end{itemize}

\section{Results}

A first partial solution to Problem \ref{ttt} is a straightforward
consequence of the Hardy--Littlewood inequalities. The Hardy--Littlewood
inequalities \cite{dimant,hl,pra} for $m$--linear forms assert that for any integer $m\geq2$ there exist constants $C_{m,p}^{\mathbb{K}},D_{m,p}^{\mathbb{K}
}\geq1 $ such that
\begin{equation}  \label{poi}
\left\vert 
\begin{array}{l}
\displaystyle \left(\sum_{j_{1},\ldots,j_{m}=1}^{n}\left\vert
T(e_{j_{1}},\ldots,e_{j_{m}})\right\vert ^{\frac{p}{p-m}}\right) ^{\frac{p-m}{p}}\leq D_{m,p}^{\mathbb{K}}\left\Vert T\right\Vert \text{ for }m<p\leq 2m, \\
\displaystyle\left( \sum_{j_{1},\ldots,j_{m}=1}^{n}\left\vert
T(e_{j_{1}},\ldots,e_{j_{m}})\right\vert ^{\frac{2mp}{mp+p-2m}}\right) ^{\frac{mp+p-2m}{2mp}}\leq C_{m,p}^{\mathbb{K}}\left\Vert T\right\Vert \text{for }p\geq2m,
\end{array}
\right.
\end{equation}
for all $m$--linear forms $T:\ell_{p}^{n}\times\cdots\times\ell_{p}^{n}%
\rightarrow\mathbb{K}$, all positive integers $n$.

The optimal constants $C_{m,p}^{\mathbb{K}},D_{m,p}^{\mathbb{K}}$ are
unknown; even the asymptotic behaviour of these constants is unknown. Up to
now, the best estimates for $C_{m,p}^{\mathbb{K}}$ can be found in \cite%
{ara, jfa}:%
\begin{equation*}
C_{m,p}^{\mathbb{K}}\leq\left( \sqrt{2}\right) ^{\frac{2m(m-1)}{p}}\left(
1.3m^{0.365}\right) ^{\frac{p-2m}{p}}.
\end{equation*}
For $p>2m(m-1)^{2}$ we also know from \cite{ara} that $C_{m,p}^{\mathbb{K}%
}\leq1.3m^{0.365}$; it is not known if, in general, the same estimate is
valid for the other choices of $p$. The notation of $C_{m,p}^{\mathbb{K}%
},D_{m,p}^{\mathbb{K}}$ as the optimal constants of the Hardy--Littlewood
inequalities will be kept all along the paper.

By \eqref{poi} we easily have the following:

\begin{proposition}
Let $m,n$ be positive integers and $p\in(m,\infty]$. There are positive
constants $C_{m,p}^{\mathbb{K}},D_{m,p}^{\mathbb{K}}$ such that%
\begin{align*}
g(p) & \geq\frac{1}{D_{m,p}^{\mathbb{K}}}n^{\frac{m\left( p-m\right) }{p}} \text{ for }m<p\leq 2m, \\
g(p) & \geq\frac{1}{C_{m,p}^{\mathbb{K}}}n^{\frac{mp+p-2m}{2p}} \text{ for }p\geq2m.
\end{align*}
\end{proposition}

Among other results, the main result of the present paper shows that the above estimates are far from being precise. We will show that:

\begin{itemize}
\item The exponent $\frac{m\left( p-m\right) }{p}$ can be replaced by $\frac{%
mp+p-2m}{2p}$ in the case $m<p\leq 2m;$

\item The constants $\frac{1}{C_{m,p}^{\mathbb{K}}}$ and $\frac{1}{D_{m,p}^{%
\mathbb{K}}}$ can be replaced by $1.3m^{0.365};$

\item The inequality is also valid for $2\leq p\leq m$ with the same
constants and exponents $\frac{%
	mp+p-2m}{2p}$;

\item The above exponents $\frac{%
	mp+p-2m}{2p}$ are optimal.

\end{itemize}

Recently (see \cite{frann}), it has been shown that the constants  $D_{m,p}^{\mathbb{K}}$ have essentially a very low growth but since we now improve the associated exponents, the estimates of $D_{m,p}^{\mathbb{K}}$ are not useful here.
 
To achieve our goals, we begin by revisiting the Kahane--Salem--Zygmund inequality. It is a
probabilistic result that furnishes unimodular multilinear forms with
\textquotedblleft small\textquotedblright\ norms. This result is fundamental
to the proof of the optimality of the exponents of the Hardy--Littlewood
inequality. For $p\geq1$, the Kahane--Salem--Zygmund asserts that there
exists a $m$-linear form $A:\ell_{p}^{n}\times\cdots\times\ell_{p}^{n}%
\longrightarrow\mathbb{K}$ of the form 
\begin{equation*}
A\left( x^{\left( 1\right) },\ldots,x^{\left( m\right) }\right)
=\sum\limits_{i_{1},\ldots,i_{m}=1}^{n}\delta_{i_{1}\cdots
i_{m}}x_{i_{1}}^{\left( 1\right) }\cdots x_{i_{m}}^{\left(m\right)},
\end{equation*}
with $\delta_{i_{1}\cdots i_{m}}\in\{-1,1\},$ such that 
\begin{equation*}
\left\Vert A\right\Vert \leq C_{m}n^{\frac{1}{2}+m\left( \frac{1}{2}-\frac {1%
}{p}\right) }.
\end{equation*}
However, for $1\leq p\leq2$ a better estimate
can essentially be found in \cite{fbayart}. So, we have the following:

\begin{theorem}[Kahane--Salem--Zygmund inequality]
\label{kszi} Let $n,m$ be positive integers and $p\geq1$. Then there exists
a $m$-linear form $A:\ell_{p}^{n}\times\cdots\times\ell_{p}^{n}%
\longrightarrow\mathbb{K}$ of the form 
\begin{equation*}
A\left( x^{\left( 1\right) },\ldots,x^{\left( m\right) }\right) =
\sum\limits_{i_{1},\ldots,i_{m}=1}^{n}\delta_{i_{1}\cdots
i_{m}}x_{i_{1}}^{\left( 1\right) }\cdots x_{i_{m}}^{\left(m\right)},
\end{equation*}
with $\delta_{i_{1}\cdots i_{m}}\in\{-1,1\},$ such that 
\begin{equation*}
\left\Vert A\right\Vert \leq C_{m}n^{\max\left\{ \frac{1}{2}+m\left(\frac{1}{%
2}-\frac{1}{p}\right) ,1-\frac{1}{p}\right\}}.
\end{equation*}
\end{theorem}

We shall show that \eqref{poi} can be significantly improved when dealing
with unimodular forms. It is easy to see that our main result is a
consequence of the following theorem (see Figure \ref{ddssdds}).

Before presenting the next result, let us introduce some required definitions for their proof. Let $B_{E^{\ast }}$ be the closed unit ball of the topological dual of $E$. For $s\geq 1$ we represent by $\ell _{s}^{w}(E)$ the linear space of the sequences $\left(x_{j}\right)_{j=1}^{\infty}$ in $E$ such that $\left( \varphi \left( x_{j}\right)\right) _{j=1}^{\infty }\in \ell _{s}$ for every continuous linear functional $\varphi :E\rightarrow \mathbb{K}$. For $(x_j)_{j=1}^{\infty}\in \ell_{s}^w(E)$, the expression $\sup_{\varphi \in B_{E^{\ast }}}\left( \sum_{j=1}^{\infty} | \varphi (x_{j}) |^s\right)^{\frac{1}{s}}$ defines a norm on $\ell _{s}^{w}(E)$. For $p,q\in\lbrack1,+\infty)$, a multilinear operator $T:E_{1}\times\cdots\times E_{m}\rightarrow \mathbb{K}$ is multiple $(q;p)$-summing if there exist a constant $C>0$ such that
	\[
	\left( \sum\limits_{j_{1},...,j_{m}=1}^{\infty}| T( x_{j_{1}}^{(1)},\dots,x_{j_{m}}^{(m)})|^{q}\right)^{\frac{1}{q}}\leq C \left( \sup_{\varphi \in B_{E^{\ast }}}\left( \sum_{j=1}^{\infty} | \varphi (x_{j}^{(k)})|^p\right)^{\frac{1}{p}}\right)^m
	\]
	for all $(x_{j}^{(k)})_{j=1}^{\infty}\in\ell_{p}^{w}\left( E_{k}\right)$.
For recent results of multiple summing operators we refer to \cite{popa}.

\begin{theorem}
\label{ha}If $m,n$ are positive integers and $p\in \left(\frac{2m}{m+1},\infty \right],$ then
\[
\left( \sum_{j_{1},\ldots ,j_{m}=1}^{n}\left\vert
T(e_{j_{1}},\ldots ,e_{j_{m}})\right\vert ^{\frac{2mp}{mp+p-2m}}\right) ^{\frac{mp+p-2m}{2mp}}\leq 1.3m^{0.365}\left\Vert T\right\Vert
\]
for all unimodular $m$-linear forms $T:\ell _{p}^{n}\times \cdots \times \ell _{p}^{n}\rightarrow \mathbb{K}$. Moreover, the exponent is sharp for $p\geq 2$. For $1<p\leq \frac{2m}{m+1}$ the optimal exponent is not smaller than $\frac{mp}{p-1}$ and for $\frac{2m}{m+1}<p\leq 2$ the optimal exponent belongs to $\left[\frac{mp}{p-1},\frac{2mp}{mp+p-2m}\right]$.
\end{theorem}

\begin{proof}
Using the isometric characterization of
the spaces of weak $1$-summable sequences on $c_{0}$ (see \cite{diestel}) we
know that every continuous $m$-linear form is multiple $\left( \frac{2m}{m+1}%
;1\right) $-summing with constant dominated by $1.3m^{0.365}$.

Thus 
\begin{equation*}
\left(\sum_{j_{1},\ldots,j_{m}=1}^{n}\left\vert
T(e_{j_{1}},\ldots,e_{j_{m}})\right\vert ^{\frac{2m}{m+1}}\right) ^{\frac{m+1%
}{2m}}\leq1.3m^{0.365}\left\Vert T\right\Vert \left( \sup_{\varphi\in
B_{\ell_{p^{\ast}}^{n}}}\sum_{j=1}^{n}\left\vert \varphi_{j}\right\vert
\right) ^{m}
\end{equation*}
for all $m$-linear forms 
\begin{equation*}
T:\ell_{p}^{n}\times\cdots\times\ell_{p}^{n}\rightarrow\mathbb{K}\text{.}
\end{equation*}
Hence%
\begin{equation*}
\left( n^{m}\right) ^{\frac{m+1}{2m}}\leq1.3m^{0.365}\left\Vert T\right\Vert
\left( n\frac{1}{n^{1/p^{\ast}}}\right) ^{m}
\end{equation*}
and finally%
\begin{equation*}
\left\Vert T\right\Vert \geq\frac{1}{1.3m^{0.365}}n^{\frac{mp+p-2m}{2p}}
\end{equation*}
and this means that 
\begin{equation*}
\left( \sum_{j_{1},\ldots,j_{m}=1}^{n}\left\vert
T(e_{j_{1}},\ldots,e_{j_{m}})\right\vert ^{\frac{2mp}{mp+p-2m}}\right) ^{%
\frac{mp+p-2m}{2mp}}\leq1.3m^{0.365}\left\Vert T\right\Vert .
\end{equation*}

Let us prove the optimality of the exponents for $p\geq 2$. Suppose that the theorem is valid for an exponent $r$, i.e.,
\[
\left( \sum_{j_{1},\ldots,j_{m}=1}^{n}\left\vert
T(e_{j_{1}},\ldots,e_{j_{m}})\right\vert ^{r}\right) ^{\frac{1}{r}}\leq 1.3m^{0.365}\left\Vert T\right\Vert.
\]
Since $p\geq 2$, from the Kahane--Salem--Zygmund inequality (Theorem \ref{kszi}) we have 
\[
n^{\frac{m}{r}}\leq 1.3m^{0.365}C_{m}n^{\frac{1}{2}+m\left(\frac{1}{2}-\frac{1}{p}\right)}=C_m1.3m^{0.365}n^{\frac{mp+p-2m}{2p}}
\]
and thus, making $n\rightarrow \infty $, we obtain $r\geq \frac{2mp}{mp+p-2m}$.

For $1<p\leq 2$, if the inequality holds for a certain exponent $r$, from the Kahane--Salem--Zygmund inequality (Theorem \ref{kszi}) we have
\[
n^{\frac{m}{r}}\leq Cn^{1-\frac{1}{p}}=Cn^{\frac{p-1}{p}}
\]
and thus, making $n\rightarrow \infty $, we obtain $r\geq \frac{mp}{p-1}$.

\end{proof}

\begin{figure}[!h]
\begin{tikzpicture}
\path[draw,shade,bottom color=red,opacity=.2]
(72/21,6) --
(60/17,5) -- (48/13,4) -- (42/11,7/2) --
(4,3) -- (28/8,7/2) -- (16/5,4) -- (20/7,5) -- (24/9,6) -- cycle;

\path[draw,shade,bottom color=blue,opacity=.2]
(2,0) -- (2,6) -- (24/9,6) --
(20/7,5) -- (16/5,4) -- (28/8,7/2) --
(4,3) -- (9/2,27/12) -- (5,15/8) -- (6,18/12) -- (7,21/16) -- (8,24/20) -- (9,27/24) -- (10,30/28) -- (11,33/32) -- (12,36/36) --
(12,0) -- cycle;

\path[draw,shade,bottom color=blue,opacity=.2]
(4,-1) -- (4,-1.5) -- (4.5,-1.5) -- (4.5,-1) -- cycle;
\draw (4.5,-1.3) node[right] {Non-admissible exponents};

\path[draw,shade,top color=green,opacity=.2]
(72/21,6) --
(60/17,5) -- (48/13,4) -- (42/11,7/2) --
(4,3) --
(9/2,27/12) -- (5,15/8) -- (6,18/12) -- (7,21/16) -- (8,24/20) -- (9,27/24) -- (10,30/28) -- (11,33/32) -- (12,36/36) --
(12,6) -- cycle;

\path[draw,shade,top color=green,opacity=.2]
(4,-1.75) -- (4,-2.25) -- (4.5,-2.25) -- (4.5,-1.75) -- cycle;
\draw (4.5,-2.05) node[right] {Admissible exponents};

\path[draw,shade,top color=red,opacity=.2]
(4,-2.5) -- (4,-3) -- (4.5,-3) -- (4.5,-2.5) -- cycle;
\draw (4.5,-2.8) node[right] {Unknown exponents};

\draw (3,0) node[below] {$\frac{2m}{m+1}$};
\draw[dotted] (3,0) -- (3,6);

\draw (4,0) node[below] {$2$};
\draw[dotted] (4,0) -- (4,3);

\draw (2,3) node[left] {$2m$};
\draw[dotted] (2,3) -- (4,3);

\draw (2,3/4) node[left] {$\frac{2m}{m+1}$};
\draw[dotted] (2,3/4) -- (12,3/4);

\draw (2,0) node[below] {$1$};
\draw[dotted] (2,6) -- (2,0);

\draw[->] (2,0) -- (12.5,0) node[below] {$p$};
\draw[->] (2,0) -- (2,6.5) node[left] {$ $};
\draw (2,0) node[left] {$0$};


\begin{scope}[shift={(0,0)}]
\draw[domain=24/9:4, ultra thick, smooth]
plot (\x,{(1.5*\x)/(\x-2)});
\end{scope}

\draw (2.5,5) node[below] {$\frac{mp}{p-1}$};

\begin{scope}[shift={(0,0)}]
\draw[domain=72/21:12, ultra thick, smooth]
plot (\x,{(1.5*\x)/(2*\x-6)});
\end{scope}

\draw (5.5,3) node[below] {$\frac{2mp}{mp+p-2m}$};
\end{tikzpicture}

\caption{Graphical overview of the exponents in Theorem \ref{ha}.}\label{ddssdds}

\end{figure}
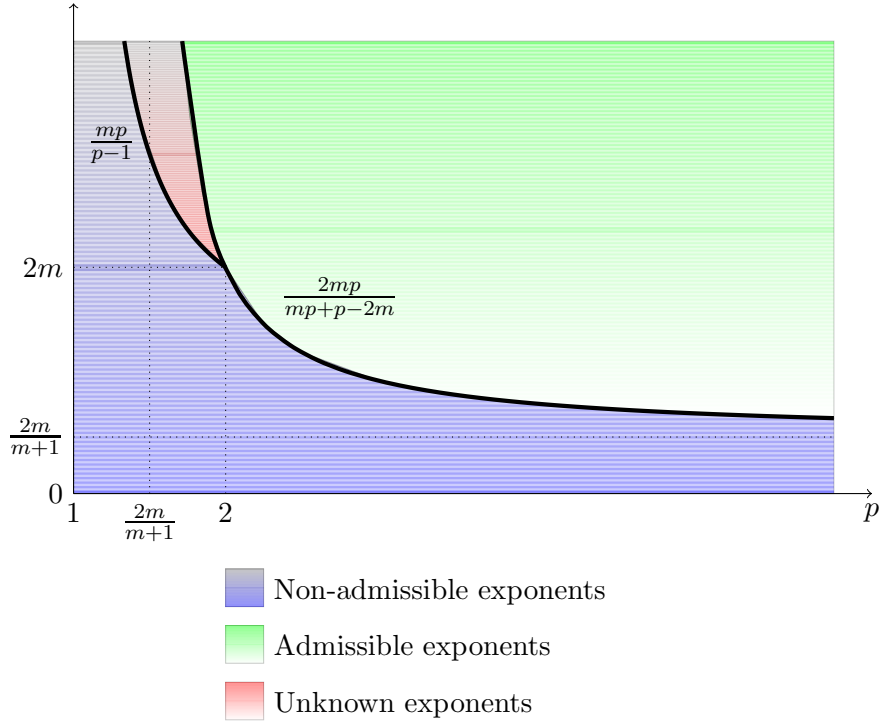

The determination of the unknown exponents rely in an open result on the interpolation of certain multilinear forms, which seems to be open for a long time: every continuous $m$-linear form from $\ell_1 \times \cdots \ell_1$ to $\mathbb{K}$ is multiple $(1,1)$-summing and every continuous $m$-linear operators from $\ell_2 \times \cdots \ell_2$ to $\mathbb{K}$ is multiple $(\frac{2m}{m+1},1)$-summing. What about intermediate results for $\ell_p$. The natural result would be, for $1\leq p\leq 2$ that every continuous $m$-linear operators from $\ell_p \times \cdots \ell_p$ to $\mathbb{K}$ is multiple $(\frac{mp}{m+p-1},1)$-summing. Even in the linear case, similar vector-valued problems remain open (see \cite{bennett})

We conjecture the following optimal result:

\begin{conjecture}
\label{conjecture}If $m,n$ are positive integers and $p\in\lbrack1,\infty],$ then
there is a constant $K_{m}$ such that
\[
\left\vert
\begin{array}{l}
\displaystyle\left(  \sum_{j_{1},\ldots,j_{m}=1}^{n}\left\vert T(e_{j_{1}},\ldots,e_{j_{m}})\right\vert ^{\frac{mp}{p-1}}\right)  ^{\frac{p-1}{mp}}\leq K_{m}\left\Vert T\right\Vert \text{ for }1 \leq p\leq2,\\
\displaystyle\left(\sum_{j_{1},\ldots,j_{m}=1}^{n}\left\vert T(e_{j_{1}},\ldots,e_{j_{m}})\right\vert ^{\frac{2mp}{mp+p-2m}}\right)  ^{\frac{mp+p-2m}{2mp}}\leq 1.3m^{0.365}\left\Vert T\right\Vert \text{ for }p\geq2,
\end{array}
\right.
\]
for all unimodular $m$-linear forms $T:\ell_{p}^{n}\times\cdots\times\ell_{p}^{n}\rightarrow\mathbb{K}$ and the exponents are sharp.
\end{conjecture}

\section{Revisiting the classical unbalancing lights problem}

\subsection{The classical unbalancing lights problem}

In this section we prove a non asymptotic version of \eqref{6hh3} showing the only situations in which the minimum estimate is achieved.

\begin{theorem}
\label{611}Let $a_{ij}=\pm1$ for $1\leq i,j\leq n$. Then there exist $x_{i},y_{j}=\pm1$, $1\leq i,j\leq n$, such that 
\begin{equation*}
\sum\limits_{i,j=1}^{n}a_{ij}x_{i}y_{j}\geq2^{-1/2}n^{3/2},
\end{equation*}
and the equality happens if, and only if, $n=2$ and 
\begin{equation}
\left( a_{ij}\right) =\pm
\begin{bmatrix}
1 & 1 \\ 
1 & -1%
\end{bmatrix}
,\pm%
\begin{bmatrix}
1 & 1 \\ 
-1 & 1%
\end{bmatrix}
,\pm%
\begin{bmatrix}
1 & -1 \\ 
1 & 1%
\end{bmatrix}
\text{ or }\pm%
\begin{bmatrix}
-1 & 1 \\ 
1 & 1%
\end{bmatrix}
.  \label{cf}
\end{equation}
In other words, for any initial configuration $\left( a_{ij}\right) $ it is
possible to perform switches so that the number of lights on minus the
number of lights off is at least $2^{-1/2}n^{3/2}$ and the equality happens
if and only if $\left( a_{ij}\right) $ is as in (\ref{cf}).
\end{theorem}

\begin{proof}
Littlewood's $4/3$-inequality asserts that 
\begin{equation}
\left( \sum_{j,k=1}^{n}\left\vert T(e_{j},e_{k})\right\vert ^{\frac{4}{3}%
}\right) ^{\frac{3}{4}}\leq\sqrt{2}\sup_{\left\Vert x\right\Vert ,\left\Vert
y\right\Vert \leq1}\left\vert T(x,y)\right\vert ,  \label{999}
\end{equation}
for all continuous bilinear forms $T\colon\ell_{\infty}^{n}\times\ell_{%
\infty }^{n}\rightarrow\mathbb{R}$ and all positive integers $n$. It is not
difficult to prove that the supremum in the right-hand-side of (\ref{999})
is achieved in the extreme points of the closed unit ball of $%
\ell_{\infty}^{n}$. Since these extreme point are precisely those with the
entries $1$ or $-1,$ we conclude that there exist $x_{i},y_{j}=\pm1$, $1\leq
i,j\leq n$, such that 
\begin{equation*}
\sum\limits_{i,j=1}^{n}a_{ij}x_{i}y_{j}\geq2^{-1/2}n^{3/2}.
\end{equation*}
It remains to prove that the equality happens if and only if $\left(
a_{ij}\right) $ is as in (\ref{cf}). To prove this we recall the following
result of \cite{pt}:

\begin{itemize}
\item A bilinear form $T$ is an (norm-one) extreme of Littlewood's $4/3$
inequality if and only if $T$ is written as 
\begin{align*}
T(x,y) & =\pm2^{-1/2}\left(
x_{i_{1}}y_{i_{2}}+x_{i_{1}}y_{i_{3}}+x_{i_{4}}y_{i_{2}}-x_{i_{4}}y_{i_{3}}%
\right) , \\
T(x,y) & =\pm2^{-1/2}\left(
x_{i_{1}}y_{i_{2}}+x_{i_{1}}y_{i_{3}}-x_{i_{4}}y_{i_{2}}+x_{i_{4}}y_{i_{3}}%
\right) , \\
T(x,y) & =\pm2^{-1/2}\left(
x_{i_{1}}y_{i_{2}}-x_{i_{1}}y_{i_{3}}+x_{i_{4}}y_{i_{2}}+x_{i_{4}}y_{i_{3}}%
\right) , \\
T(x,y) & =\pm2^{-1/2}\left(
-x_{i_{1}}y_{i_{2}}+x_{i_{1}}y_{i_{3}}+x_{i_{4}}y_{i_{2}}+x_{i_{4}}y_{i_{3}}%
\right)
\end{align*}
for $i_{1}\neq i_{4}$ and $i_{2}\neq i_{3}.$
\end{itemize}

From the above theorem we conclude that when we deal with bilinear forms
with coefficients $1$ or $-1$, the equality in (\ref{999}) happens if and
only if $n=2$ and 
\begin{align*}
T(x,y) & =\pm\left( x_{1}y_{1}+x_{1}y_{2}+x_{2}y_{1}-x_{2}y_{2}\right) , \\
T(x,y) & =\pm\left( x_{1}y_{1}+x_{1}y_{2}-x_{2}y_{1}+x_{2}y_{2}\right) , \\
T(x,y) & =\pm\left( x_{1}y_{1}-x_{1}y_{2}+x_{2}y_{1}+x_{2}y_{2}\right) , \\
T(x,y) & =\pm\left( -x_{1}y_{1}+x_{1}y_{2}+x_{2}y_{1}+x_{2}y_{2}\right)
\end{align*}
and the proof is done.
\end{proof}

\subsection{The classical unbalancing lights problem in higher dimensions}

The next result provides an asymptotic variant of \eqref{6hh2} in the lines of \eqref{6hh3}:

\begin{theorem}
\label{654} Let $m$ be a positive integer and $a_{i_{1}\cdots i_{m}}=\pm1$
for all $i_{1},\ldots,i_{m}$. Then, for all $k=1,\ldots,m,$ there exist $%
x_{i_{j}}^{(k)}=\pm1$, $1\leq j\leq n$, such that 
\begin{equation}  \label{a1a}
\sum\limits_{i_{1},\ldots,i_{m}=1}^{n}a_{i_{1}\cdots
i_{m}}x_{i_{1}}^{(1)}\cdots x_{i_{m}}^{(m)}\geq\left( 2^{1-\psi\left(
m+1\right) -\gamma}\left(\prod\limits_{k=2}^{m}\frac{\Gamma\left( \frac{3k-2%
}{2k}\right)}{\Gamma\left( \frac{3}{2}\right) }\right) +o(1)\right) n^{\frac{%
m+1}{2}},
\end{equation}
where $\psi$ is the digamma function and $\gamma$ is the Euler-Mascheroni
constant.
\end{theorem}

We begin by recalling some useful technical results:

\begin{lemma}[Minkowski]
\label{min}If $0<p<q<\infty$, then%
\begin{equation*}
\left( \sum\limits_{j=1}^{n}\left( \sum\limits_{i=1}^{n}\left\vert
a_{ij}\right\vert ^{p}\right) ^{\frac{1}{p}q}\right) ^{\frac{1}{q}%
}\leq\left( \sum\limits_{i=1}^{n}\left( \sum\limits_{j=1}^{n}\left\vert
a_{ij}\right\vert ^{q}\right) ^{\frac{1}{q}p}\right) ^{\frac{1}{p}}
\end{equation*}
for all positive integers $n$ and all scalars $a_{ij}.$
\end{lemma}

\begin{lemma}[Haagerup, see \protect\cite{nazarov}]
\label{haa}Let $1\leq p\leq2.$ For all sequence of real scalars $(a_{i})$ we
have 
\begin{equation*}
\left( \sum\limits_{i=1}^{n}\left\vert a_{i}\right\vert ^{2}\right)
^{1/2}\leq\left( \left( \frac{2^{\frac{p-2}{2}}\Gamma\left( \frac{p+1}{2}%
\right) }{\Gamma\left( \frac{3}{2}\right) }\right) ^{-1}+o(1)\right) \left(
\int\limits_{0}^{1}\left\vert \sum\limits_{k=1}^{n}r_{i}(t)a_{i}\right\vert
^{p}dt\right) ^{\frac{1}{p}}.
\end{equation*}
\end{lemma}

The next lemma is a well-known consequence of the Krein--Milman Theorem:

\begin{lemma}
\label{krein}For all $m$-linear forms $A:\ell_{\infty}^{n}\times\cdots
\times\ell_{\infty}^{n}\rightarrow\mathbb{R}$ we have%
\begin{equation*}
\left\Vert A\right\Vert =\max\left\vert A\left(
x^{(1)},\ldots,x^{(m)}\right) \right\vert ,
\end{equation*}
where $x^{(j)}$ has all entries equal to $1$ or $-1$, for all $j=1,\ldots,m.$
\end{lemma}

Now we are able to begin the proof. Let%
\begin{equation*}
f(p):=\left( \frac{2^{\frac{p-2}{2}}\Gamma\left( \frac{p+1}{2}\right) }{%
\Gamma\left( p\right) }\right) ^{-1}.
\end{equation*}
Consider the $m$-linear form 
\begin{equation*}
A\left( x^{(1)},\ldots,x^{(m)}\right) =\sum\limits_{i,j=1}^{n}a_{i_{1}\cdots
i_{m}}x_{i_{1}}^{(1)}\cdots x_{i_{m}}^{(m)}.
\end{equation*}
For bilinear forms, using Lemma \ref{haa}, we have%
\begin{align}  \label{n8}
\sum\limits_{j=1}^{n}\left( \sum\limits_{i=1}^{n}\left\vert
a_{ij}\right\vert ^{2}\right) ^{1/2} & =\sum\limits_{j=1}^{n}\left(
\sum\limits_{i=1}^{n}\left\vert A\left( e_{i},e_{j}\right) \right\vert
^{2}\right) ^{1/2} \\
& \leq\left( f(1)+o(1)\right)
\sum\limits_{j=1}^{n}\int\limits_{0}^{1}\left\vert
\sum\limits_{i=1}^{n}r_{i}(t)A\left( e_{i},e_{j}\right) \right\vert dt 
\notag \\
& \leq\left( f(1)+o(1)\right)
\sup_{t\in\lbrack0,1]}\sum\limits_{j=1}^{n}\left\vert A\left(
\sum\limits_{i=1}^{n}r_{i}(t)e_{i},e_{j}\right) \right\vert  \notag \\
& \leq\left( f(1)+o(1)\right) \left\Vert A\right\Vert .  \notag
\end{align}
and, by symmetry and by Lemma \ref{min} we have%
\begin{equation}  \label{n9}
\left( \sum\limits_{j=1}^{n}\left( \sum\limits_{i=1}^{n}\left\vert
a_{ij}\right\vert \right) ^{2}\right) ^{1/2}\leq\left( \left( 2^{-1/2}\frac{%
\Gamma(1)}{\Gamma\left( 3/2\right) }\right) ^{-1}+o(1)\right) \left\Vert
A\right\Vert .
\end{equation}
By the H\"{o}lder inequality for mixed sums combined with \eqref{n8} and %
\eqref{n9}, we have 
\begin{equation*}
\left( \sum_{i,j=1}^{n}\left\vert a_{ij}\right\vert ^{\frac{4}{3}}\right) ^{%
\frac{3}{4}}\leq\left( f(1)+o(1)\right) \left\Vert A\right\Vert .
\end{equation*}
For trilinear forms we have%
\begin{align}
& \left( \sum\limits_{k=1}^{n}\left( \sum\limits_{i,j=1}^{n}\left\vert
a_{ijk}\right\vert ^{2}\right) ^{\frac{1}{2}\times\frac{4}{3}}\right) ^{%
\frac{3}{4}}  \label{aa} \\
& \leq\left( f(4/3)+o(1)\right) \left(
\sum\limits_{k=1}^{n}\int\limits_{0}^{1}\left\vert
\sum\limits_{k=1}^{n}r_{k}(t)A\left( e_{i},e_{j},e_{k}\right) \right\vert ^{%
\frac{4}{3}}dt\right) ^{\frac{3}{4}}  \notag \\
& \leq\left( f(4/3)+o(1)\right) \left( f(1)+o(1)\right) \left\Vert
A\right\Vert  \notag \\
& =\left( f(1)f(4/3)+o(1)\right) \left\Vert A\right\Vert .  \notag
\end{align}
Using symmetry and Lemma \ref{min} we have%
\begin{equation}
\left( \sum\limits_{k,i=1}^{n}\left( \sum\limits_{j=1}^{n}\left\vert
a_{ijk}\right\vert ^{\frac{4}{3}}\right) ^{\frac{3}{4}\times2}\right) ^{%
\frac{1}{2}}\leq\left( f(1)f(4/3)+o(1)\right) \left\Vert A\right\Vert
\label{bb}
\end{equation}
and%
\begin{equation}
\left( \sum\limits_{k=1}^{n}\left( \sum\limits_{i=1}^{n}\left(
\sum\limits_{j=1}^{n}\left\vert a_{ijk}\right\vert ^{2}\right) ^{\frac{1}{2}%
\times\frac{4}{3}}\right) ^{\frac{3}{4}\times2}\right) ^{\frac{1}{2}%
}\leq\left( f(1)f(4/3)+o(1)\right) \left\Vert A\right\Vert .  \label{cc}
\end{equation}
By the H\"{o}lder inequality for mixed sums and \eqref{aa}, \eqref{bb}, \eqref{cc} we get
\begin{align*}
\left( \sum\limits_{i,j,k=1}^{n}\left\vert a_{ijk}\right\vert ^{3/2}\right)
^{\frac{2}{3}} & \leq\left( f(4/3)+o(1)\right) \left( f(1)+o(1)\right)
\left\Vert A\right\Vert \\
& =\left( f(1)f(4/3)+o(1)\right) \left\Vert A\right\Vert .
\end{align*}
Following this vein, for the general case we have%
\begin{align*}
\left( \sum\limits_{i_{1},\ldots,i_{m}=1}^{n}\left\vert a_{i_{1}\cdots
i_{m}}\right\vert ^{\frac{2m}{m+1}}\right) ^{\frac{m+1}{2m}} & \leq
\prod\limits_{k=2}^{m}\left( f\left( \frac{2\left( k-1\right) }{k}%
\right)+o(1)\right) \left\Vert A\right\Vert \\
& =\left( \left( \prod\limits_{k=2}^{m}f\left( \frac{2\left( k-1\right)}{k}%
\right) \right) +o(1)\right) \left\Vert A\right\Vert.
\end{align*}
We thus conclude that there exist $x_{i_{j}}=\pm1$, $1\leq j\leq n$, such
that 
\begin{align*}
\sum\limits_{i_{1},\ldots,i_{m}=1}^{n}a_{i_{1}\cdots
i_{m}}x_{i_{1}}^{(1)}\cdots x_{i_{m}}^{(m)} & \geq\left( \left(
\prod\limits_{k=2}^{m}f\left(\frac{2\left( k-1\right) }{k}\right) \right)
^{-1}+o(1)\right)n^{\frac{m+1}{2}} \\
& =\left( \prod\limits_{k=2}^{m}\left( \frac{2^{-1/k}\Gamma\left(\frac{3k-2}{%
2k}\right) }{\Gamma\left( \frac{3}{2}\right) }\right)+o(1)\right) n^{\frac{%
m+1}{2}} \\
&=\left( 2^{1-\psi\left( m+1\right) -\gamma}\left(\prod\limits_{k=2}^{m}%
\frac{\Gamma\left( \frac{3k-2}{2k}\right) }{\Gamma\left( \frac{3}{2}\right) }%
\right) +o(1)\right) n^{\frac{m+1}{2}},
\end{align*}
where $\psi$ is the digamma function and $\gamma$ is the Euler-Mascheroni
constant. The optimality of the exponent $\frac{m+1}{2}$ can be proved, as
usual, using the Kahane--Salem--Zygmund inequality.

Observing that Lemma \ref{haa} holds for all sequence of real scalars $%
(a_{i}),$ the argument of the previous section can be adapted to prove the
following version, with asymptotic constants, of the Bohnenblust--Hille
inequality:

\begin{theorem}
For all continuous $m$-linear forms $T:c_{0}\times\cdots\times
c_{0}\rightarrow\mathbb{R}$ we have 
\begin{equation}  \label{ne}
\left(\sum\limits_{i_{1},\ldots,i_{m}=1}^{n}\left\vert T\left(
e_{i_{1}},\ldots,e_{i_{m}}\right) \right\vert ^{\frac{2m}{m+1}}\right) ^{%
\frac{m+1}{2m}}\leq\left( \frac{1}{2^{1-\psi\left(m+1\right) -\gamma}}\left(
\prod\limits_{k=2}^{m}\frac{\Gamma\left(\frac{3}{2}\right)}{\Gamma\left(%
\frac{3k-2}{2k}\right)}\right)+o(1)\right) \left\Vert T\right\Vert.
\end{equation}
\end{theorem}

\medskip

\begin{center}
\begin{tabular}{|c|c|}
\hline
& Value of $\frac{1}{2^{1-\psi\left( m+1\right) -\gamma}}\left(
\prod\limits_{k=2}^{m}\frac{\Gamma\left( \frac{3}{2}\right) }{\Gamma\left( 
\frac{3k-2}{2k}\right) }\right) $ \\ \hline
$m=2$ & $\sqrt{\pi/2}\approx1.2533$ \\ \hline
$m=5$ & $1.9895$ \\ \hline
$m=10$ & $3.0555$ \\ \hline
$m=100$ & $15.2457$ \\ \hline
$m=1000$ & $81.1974$ \\ \hline
\end{tabular}
\end{center}

\medskip

From \eqref{ne} and repeating the proof of Theorem \ref{ha} we
have:

\begin{theorem}
Let $p\in\lbrack2,\infty]$. For all unimodular $m$-linear forms $T:\ell
_{p}^{n}\times\cdots\times\ell_{p}^{n}\rightarrow\mathbb{R}$ we have%
\begin{equation*}
\left(\sum\limits_{i_{1},\ldots,i_{m}=1}^{n}\left\vert T\left(
e_{i_{1}},\ldots,e_{i_{m}}\right) \right\vert ^{\frac{2mp}{mp+p-2m}}\right)
^{\frac{mp+p-2m}{2mp}}\leq\left( \frac{1}{2^{1-\psi\left( m+1\right) -\gamma}%
}\left( \prod\limits_{k=2}^{m}\frac{\Gamma\left( \frac{3}{2}\right) }{%
\Gamma\left( \frac{3k-2}{2k}\right) }\right) +o(1)\right) \left\Vert
T\right\Vert.
\end{equation*}
\end{theorem}

\section{Blow up rate for the Hardy--Littlewood inequalities for unimodular forms}

In this section we provide the blow up rate for the constants in Theorem \ref{ha} as $n$ grows when the $\ell_{\frac{2mp}{mp+p-2m}}$-norm in the left-hand-side is replaced by an $\ell_r$-norm with $0 < r < \infty$. More precisely, we prove the following result:

\begin{theorem}
If $m$ is a positive integers and $(r,p)\in (0,\infty)\times \left(\frac{2m}{m+1},\infty\right]$ then
\[
\left(\sum_{j_{1},\ldots,j_{m}=1}^{n}\left\vert
T(e_{j_{1}},\ldots,e_{j_{m}})\right\vert ^{r}\right) ^{\frac{1}{r}}\leq 1.3m^{0.365}n^{\max\{\frac{2mr+2mp-mpr-pr}{2pr},0\}}\left\Vert T\right\Vert
\]
for all unimodular $m$-linear forms $T:\ell_{p}^{n}\times\cdots\times\ell_{p}^{n}\rightarrow\mathbb{K}$ and all positive integers $n$. Moreover, for $(r,p)$ belonging to $\left(\left(0,\frac{2mp}{mp+p-2m}\right)\times \left[2,\infty\right]\right)\cup\left( \left[\frac{2mp}{mp+p-2m},\infty\right)\times \left(\frac{2m}{m+1},\infty\right]\right)$ the power $\max\{\frac{2mr+2mp-mpr-pr}{2pr},0\}$ is sharp and for $(r,p)$ belonging to $\left(0,\frac{2mp}{mp+p-2m}\right)\times \left(\frac{2m}{m+1},2\right)$ the optimal exponent of $n$ belongs to the interval  $\left[\max\{\frac{mp+r-pr}{pr},0\},\frac{2mr+2mp-mpr-pr}{2pr}\right]$.
\end{theorem}

\begin{proof}
For $p>\frac{2m}{m+1}$ we know from Theorem \ref{ha} that 
\begin{equation}
\left( \sum_{j_{1},\ldots,j_{m}=1}^{n}|T(e_{j_{1}},\ldots,e_{j_{m}})| ^{\frac{2mp}{mp+p-2m}}\right)^{\frac{mp+p-2m}{2mp}}\leq 1.3m^{0.365}\left\Vert
T\right\Vert.  \label{43333}
\end{equation}
Therefore, if $(r,p) \in \left(0,\frac{2mp}{mp+p-2m}\right)\times \left(\frac{2m}{m+1},\infty\right]$, from H\"{o}lder's inequality and \eqref{43333} we have 
\begin{align*}
& \left( \sum_{j_{1},\ldots,j_{m}=1}^{n}\left\vert T(e_{j_{1}},\ldots,e_{j_{m}})\right\vert ^{r}\right) ^{\frac{1}{r}} \\
& \leq \left( \sum\limits_{j_{1},\ldots,j_{m}=1}^{n}\left\vert T(e_{j_{1}},\ldots,e_{j_{m}})\right\vert ^{\frac{2mp}{mp+p-2m}}\right) ^{\frac{mp+p-2m}{2mp}} \left(\sum_{j_{1},\ldots,j_{m}=1}^{n}\left\vert 1\right\vert ^{\frac{2mpr}{2mp+2mr-mpr-pr}}\right) ^{\frac{2mp+2mr-mpr-pr}{2mpr}} \\
& \leq 1.3m^{0.365}\left\Vert T\right\Vert \left( n^{m}\right) ^{\frac{2mp+2mr-mpr-pr}{2mpr}} \\
& = 1.3m^{0.365} n^{\frac{2mr+2mp-mpr-pr}{2pr}}\left\Vert T\right\Vert.
\end{align*}

Let us prove the optimality of the exponents for $(r,p) \in \left(0,\frac{2mp}{mp+p-2m}\right)\times \left[2,\infty\right]$. Suppose that the theorem is valid for an exponent $s$, i.e., 
\[
\left( \sum_{j_{1},\ldots,j_{m}=1}^{n}\left\vert
T(e_{j_{1}},\ldots,e_{j_{m}})\right\vert ^{r}\right) ^{\frac{1}{r}}\leq 1.3m^{0.365}n^{s}\left\Vert T\right\Vert .
\]
Since $p\geq 2$, from the Kahane--Salem--Zygmund inequality (Theorem \ref{kszi}) we have 
\begin{equation*}
n^{\frac{m}{r}}\leq 1.3m^{0.365}n^{s}C_{m}n^{\frac{1}{2}+m\left(\frac{1}{2}-\frac{1}{p}\right)}=C_m1.3m^{0.365}n^{s+\frac{mp+p-2m}{2p}}
\end{equation*}
and thus, making $n\rightarrow \infty $, we obtain $s\geq \frac{2mr+2mp-mpr-pr}{2pr}$.

If $\left(r,p\right) \in \left[\frac{2mp}{mp+p-2m},\infty \right)\times \left(\frac{2m}{m+1},\infty\right]$ we have $\frac{2mr+2mp-mpr-pr}{2pr} \leq 0$ and 
\begin{align*}
\left( \sum_{j_{1},\ldots,j_{m}=1}^{n}\left\vert
T(e_{j_{1}},\ldots,e_{j_{m}})\right\vert ^{r}\right) ^{\frac{1}{r}} & \leq	\left(\sum_{j_{1},\ldots,j_{m}=1}^{n}\left\vert
T(e_{j_{1}},\ldots,e_{j_{m}})\right\vert ^{\frac{2mp}{mp+p-2m}}\right) ^{%
\frac{mp+p-2m}{2mp}} \\
& \leq 1.3m^{0.365}\left\Vert T\right\Vert \\
& = 1.3m^{0.365}n^{\max \left\{\frac{2mr+2mp-mpr-pr}{2pr},0\right\}}\left%
\Vert T\right\Vert.
\end{align*}

In this case the optimality of the exponent $\max \left\{\frac{2mr+2mp-mpr-pr}{2pr},0\right\}$ is immediate, since no negative exponent of $n$ is possible.

If $(r,p) \in \left(0,\frac{2mp}{mp+p-2m}\right)\times \left(\frac{2m}{m+1},2\right)$, we just have an estimate for the optimal exponent of $n$. In fact, suppose that the inequalities are valid for an exponent $s\geq 0$, i.e.,
\begin{equation*}
\left( \sum_{j_{1},\ldots,j_{m}=1}^{n}\left\vert
T(e_{j_{1}},\ldots,e_{j_{m}})\right\vert ^{r}\right) ^{\frac{1}{r}}\leq
1.3m^{0.365}n^{s}\left\Vert T\right\Vert.
\end{equation*}
Since $1\leq \frac{2m}{m+1} < p\leq 2$, from the Kahane--Salem--Zygmund inequality (Theorem \ref{kszi}) we have 
\[
n^{\frac{m}{r}}\leq 1.3m^{0.365}n^{s}C_{m}n^{1-\frac{1}{p}}=1.3m^{0.365}C_{m}n^{s+\frac{p-1}{p}}
\]
and thus, making $n\rightarrow \infty $, we obtain $s\geq \frac{mp+r-pr}{pr}$.
\end{proof}

If Conjecture \ref{conjecture} is correct, using the same ideas of the proof of the previous theorem it is possible to improve it to the following optimal result:

\begin{conjecture}
If $m$ is a positive integers and $(r,p)\in (0,\infty)\times (1,\infty]$ then there is a constant $K_{m}$ such that 
\begin{equation*}
\left\vert 
\begin{array}{l}
\displaystyle\left( \sum_{j_{1},\ldots,j_{m}=1}^{n}\left\vert
T(e_{j_{1}},\ldots,e_{j_{m}})\right\vert ^{r}\right) ^{\frac{1}{r}}\leq
K_{m}n^{\max\{\frac{mp+r-pr}{pr},0\}}\left\Vert T\right\Vert\text{ for }1< p\leq2, \\
\displaystyle\left(\sum_{j_{1},\ldots,j_{m}=1}^{n}\left\vert
T(e_{j_{1}},\ldots,e_{j_{m}})\right\vert ^{r}\right) ^{\frac{1}{r}}\leq
1.3m^{0.365}n^{\max\{\frac{2mr+2mp-mpr-pr}{2pr},0\}}\left\Vert T\right\Vert\text{ for }p\geq2,
\end{array}
\right.
\end{equation*}
for all unimodular $m$-linear forms $T:\ell_{p}^{n}\times\cdots\times%
\ell_{p}^{n}\rightarrow\mathbb{K}$ and all positive integers $n$. Moreover,
the exponents $\max\{\frac{2mr+2mp-mpr-pr}{2pr},0\}$ and $\max\{\frac{mp+r-pr%
}{pr},0\}$ are sharp.
\end{conjecture}

In fact, the novelty is the case $1 < p\leq 2$. Supposing that Conjecture \ref{conjecture} is true, if $(r,p) \in \left(0,\frac{mp}{p-1}\right)\times (1,2]$, from H\"{o}lder's inequality we have
\[
\left( \sum\limits_{j_{1},\ldots,j_{m}=1}^{n}\left\vert
T(e_{j_{1}},\ldots,e_{j_{m}})\right\vert ^{r}\right) ^{\frac{1}{r}} \leq K_{m} n^{\frac{mp+r-pr}{pr}}\left\Vert T\right\Vert.
\]
On the other hand, if the above inequalities are valid for an exponent $s$ instead of $\frac{mp+r-pr}{pr}$, since $1 < p\leq 2$, from the Kahane--Salem--Zygmund inequality (Theorem \ref{kszi}) we have
\begin{equation*}
n^{\frac{m}{r}}\leq Cn^{s}n^{1-\frac{1}{p}}=Cn^{s+\frac{p-1}{p}}
\end{equation*}
and $$s\geq \frac{mp+r-pr}{pr}.$$ If $\left(r,p\right) \in \left[\frac{mp}{p-1},\infty \right)\times (1,2]$ we have $\frac{mp+r-pr}{pr}\leq 0$ and, in this case, the optimality of the exponent $\max \{\frac{mp+r-pr}{pr},0\}$ is immediate, since no negative exponent of $n$ is possible.

\bigskip

\noindent \textbf{Acknowledgement.} Part of this paper was written when D. Pellegrino was visiting Prof. F. Bayart at the Department of Mathematics of Universit\'{e} Blaise Pascal at Clermont Ferrand. D. Pellegrino thanks Prof. Bayart for important comments and thanks the R\'{e}sau Franco-Br\'{e}silian en Math\'{e}matiques for the financial support.

\end{document}